\documentclass{amsart}
\usepackage{amssymb,amsthm,amstext,amsmath,amscd,latexsym,amsfonts,amscd}
\usepackage[usenames]{color}
\usepackage{comment}
\usepackage{enumerate}
\usepackage{braket}
\usepackage[all]{xy}
\usepackage{ulem}
\usepackage{soul}
\definecolor{mildyellow}{cmyk}{0, 0, 0.5,0}
\sethlcolor{mildyellow}

\newtheorem{theorem}{Theorem}[section]
\newtheorem{proposition}[theorem]{Proposition}
\newtheorem{lemma}[theorem]{Lemma}
\newtheorem{corollary}[theorem]{Corollary}
\theoremstyle{definition}
\newtheorem{definition}[theorem]{Definition}
\newtheorem{remark}[theorem]{Remark}

\newtheorem{example}[theorem]{Example}
\theoremstyle{plain}

\newcommand{\disp}{\displaystyle}
\newcommand{\Loc}{\mathop{\mathrm{Loc}}\nolimits}
\newcommand{\id}{\mathop{\mathrm{id}}\nolimits}
\newcommand{\Ker}{\mathop{\mathrm{Ker}}\nolimits}

\newcommand{\Ext}{\mathop{\mathrm{Ext}}\nolimits}
\newcommand{\Spec}{\mathop{\mathrm{Spec}}\nolimits}
\newcommand{\Supp}{\mathop{\mathrm{Supp}}\nolimits}
\newcommand{\supp}{\mathop{\mathrm{supp}}\nolimits}

\newcommand{\depth}{\mathop{\mathrm{depth}}\nolimits}
\newcommand{\Hom}{\mathop{\mathrm{Hom}}\nolimits}
\newcommand{\RHom}{\mathop{\mathrm{RHom}}\nolimits}
\newcommand{\RGamma}{\mathop{\mathrm{R\Gamma}}\nolimits}
\newcommand{\LLambda}{\mathop{\mathrm{L\Lambda}}\nolimits}
\newcommand{\Mod}{\mathop{\mathrm{Mod}}\nolimits}
\newcommand{\Inj}{\mathop{\mathrm{Inj}}\nolimits}

\newcommand{\Image}{\mathop{\mathrm{Im}}\nolimits}

\newcommand{\Lotimes}{\mathop{\otimes_R^{\rm L}}\nolimits}

\newcommand{\fa}{\mathfrak{a}}

\newcommand{\fp}{\mathfrak{p}}

\newcommand{\fq}{\mathfrak{q}}
\newcommand{\fm}{\mathfrak{m}}
\newcommand{\cD}{\mathcal{D}}
\newcommand{\cL}{\mathcal{L}}

\newcommand{\cK}{\mathcal{K}}

\newcommand{\Z}{\mathbb{Z}}

\begin{document}

\title[A Local duality principle]{A Local duality principle in derived categories of commutative Noetherian rings}
\author[T. Nakamura]{Tsutomu Nakamura}
\address[T. Nakamura]{Graduate School of Natural Science and Technology Okayama University,
Okayama, 700-8530 Japan}
\email{t.nakamura@s.okayama-u.ac.jp}
\author[Y. Yoshino]{Yuji Yoshino}
\address[Y. Yoshino]{Graduate School of Natural Science and Technology Okayama University,
Okayama, 700-8530 Japan}
\email{yoshino@math.okayama-u.ac.jp}
\date{}
\begin{abstract}Let $R$ be a commutative Noetherian ring.
We introduce the notion of colocalization functors $\gamma_W$ with supports in arbitrary subsets $W$ of $\Spec R$. 
If W is a specialization-closed subset, then $\gamma_W$ coincides with the right derived functor $\RGamma_W$ of the section functor $\Gamma_W$ with support in $W$.
We prove that the local duality theorem and the vanishing theorem of Grothendieck type hold for $\gamma_W$ with $W$ being an arbitrary subset.
\end{abstract}
\maketitle


\section{Introduction}
\label{1}
Throughout this paper, we assume that $R$ is a commutative Noetherian ring.
We denote by $\cD=D(\Mod R)$ the derived category of complexes of $R$-modules, 
by which we mean that $\cD$ is 
the unbounded derived category. 
Neeman $\text{\cite{N}}$ proved that there is a natural one-one correspondence between the set of subsets of $\Spec R$ and the set of localizing subcategories  of $\cD$. 
We denote by $\cL_W$ the localizing subcategory  corresponding to a subset  $W$  of  $\Spec R$. 
The localization theory of triangulated categories $\text{\cite{K}}$ yields a right adjoint $\gamma_W$  to the inclusion functor $\cL_W \hookrightarrow \cD$, and such an adjoint is unique.
This functor  $\gamma_W : \cD  \to \cL_W (\hookrightarrow \cD)$ is our main target of this paper, and
we call it {\it the colocalization functor with support in  $W$}. 

If $V$ is a specialization-closed subset of $\Spec R$, then $\gamma_{V}$ is nothing but the right derived functor $\mathrm{R}\Gamma_{V}$ of the section functor  $\Gamma_V$ with support in $V$, whose $i$th right derived functor $H^i_V(-)=H^i(\mathrm{R}\Gamma_{V}(-))$ is known as the $i$th local cohomology functor. 
For a general subset $W$  of  $\Spec R$, the colocalization functor $\gamma _W$ is not necessarily a right derived functor of an additive functor defined on 
the category $\Mod R$ of $R$-modules. 

In this paper, we establish several results concerning the colocalization functor $\gamma_W$, 
where $W$ is an arbitrary subset of $\Spec R$. 
Notable are extensions of the local duality theorem and Grothendieck type vanishing theorem of local cohomology.
The local duality can be viewed as an isomorphism
$$\RGamma_{V} \RHom_R(X,Y)\cong  \RHom_R(X, \RGamma_{V}Y),$$
where $V$ is a specialization-closed subset of $\Spec R$, $X\in \cD^{-}_{\rm fg}$ and $Y\in \cD^{+}$; see \cite[Proposition 6.1]{F}. 
The following theorem generalizes this isomorphism to the case of colocalization functors $\gamma_W$.

\begin{theorem}[Theorem \ref{LD Principle}]\label{LDP Intro}
Let $W$ be a subset of $\Spec R$ and let  $X, Y \in \cD$.
We denote by $\dim W$ the supremum of the lengths of chains of prime ideals in $W$.
Suppose that one of the following conditions holds: 
\begin{enumerate}
\item[$(1)$] 
$X\in \cD^-_{\rm fg}$, $Y \in \cD^+$ and $\dim W$ is finite;
\item[$(2)$] $X\in \cD_{\rm fg}$, $Y$ is a bounded complex of injective $R$-modules and $\dim W$ is finite;
\item[$(3)$] 
$W$ is generalization-closed.
\end{enumerate}
Then there exist a natural isomorphism
$$
\gamma_W \RHom_R(X, Y)\cong \RHom_R(X,\gamma_W Y).$$
\end{theorem}\vspace{2mm}

We shall call Theorem \ref{LDP Intro}  the {\it Local Duality Principle},  which naturally implies the following corollary. 
 
\begin{corollary}[Corollary \ref{LDT}]
\label{LDT Intro}
Assume that $R$ admits a dualizing complex $D_R$.
Let $W$ be an arbitrary subset of $\Spec R$ and 
$X\in \cD_{\rm fg}$. We write $X^\dagger =\RHom_R(X,D_R)$.
Then we have a natural isomorphism 
$$\gamma_WX\cong \RHom_R(X^\dagger, \gamma_WD_R).$$
\end{corollary}
\vspace{2mm}

The local duality theorem states the validity of this isomorphism in the case that $W$ is specialization-closed, see $\text{\cite[Chapter V; Theorem 6.2]{Ha}}$ and $\text{\cite[Corollary 6.2]{F}}$.

As an application of the Local Duality Principle, we can prove the vanishing theorem of Grothendieck type for the colocalization functor $\gamma_W$ with support in an arbitrary subset $W$.
Let $\fa$ be an ideal of $R$ and $X\in \cD$.
The $\fa$-depth of $X$, which we denote by $\depth(\fa, X)$, is  
the infimum of the set $\Set{i \in \mathbb{Z}| \Ext^i_R(R/\fa,X)\neq 0}$.
More generally, for a specialization-closed subset $W$, the $W$-depth of $X$, which we denote by $\depth (W, X)$, is defined as the infimum of the set of values $\depth(\fa, X)$ for all ideals $\fa$ with $V(\fa)\subseteq W$.
When $X\in \cD_{\rm fg}$, we denote by $\dim X$ the supremum of the set $\Set{\dim H^i(X)+i| i\in \mathbb{Z}}$.

For a finitely generated $R$-module $M$, the Grothendieck vanishing theorem says that the $i$th local cohomology module $H_W ^i (M)=H^i(\RGamma_WM)$ of $M$ with support in $W$ is zero for $i<\depth (W, M)$ and $i>\dim M$.
 
We are able to generalize this theorem to the following result in \S 6. 

\begin{theorem}[Theorem \ref{GVT}]\label{GVT Intro}
Assume that $R$ admits a dualizing complex. 
Let $W$ be an arbitrary subset of $\Spec R$ with the specialization closure $\overline{W}^s$.
If $X\in \cD_{\rm fg}$, then $ H^i(\gamma_WX)=0$ unless $\depth(\overline{W}^s, X) \leq i \leq \dim X$.
\label{GVT Intro}
\end{theorem}\vspace{2mm}



In \S 3,  we give an explicit description of $\gamma _W$  for subsets $W$ of certain special type, see Theorem \ref{dimW=0 2}. 
For example, if  $W$  is a one-point set $\{ \fp\}$, then it is proved the colocalization functor  $\gamma _{\{\fp\}}$  equals $\RGamma _{V(\fp)}\RHom _R (R_{\fp},-)$, see Corollary \ref{gamma p}.
This is one of the rare cases that we know the explicit form of $\gamma _W$, while for a general subset $W$ we give in Theorem \ref{theorem1} the way how we calculate $\gamma _W$ by the induction on $\dim W$.

In \S 4, we give a complete proof of the Local Duality Principle (Theorem \ref{LDP Intro}).

The subsequent section  \S 5 is devoted to the relationship between $\gamma_W$ and left derived functors of  completion functors. 
In particular, we see that there is a subset  $W$ such that $H^i(\gamma_WI)\neq 0$ for an injective module $I$ and some $i<0$.
This observation shows that  $\gamma_W$ is not a right derived functor of an additive functor defined on $\Mod R$ in general.

In the last section \S 6, we present a precise and complete proof for Theorem \ref{GVT Intro} above.

\color{black}
\vspace{6pt}
{Acknowledgements.}  
This work was finished during the first author's visit to the Department of Mathematics at the University of Utah supported by a research grant from Research Institute for Interdisciplinary Science at Okayama University. We are grateful to Srikanth Iyengar for his helpful comments, which improve many parts of this paper.
We would also like to thank the referee for his/her careful reading.
The second author was supported by JSPS Grant-in-Aid for Scientific Research 26287008.


\color{black}
\section{Colocalization functors}\label{2}

In this section, we summarize some notion and basic facts used later in this paper. 
As in the introduction, $R$  denotes a commutative Noetherian ring and 
we work in the derived category  $\cD=D(\Mod R)$. 
Note that complexes $X$ are cohomologically indexed;
$$X=(\cdots \rightarrow X^{i-1}\rightarrow X^i\rightarrow X^{i+1}\rightarrow\cdots).$$

We denote by $\cD^+$ (resp. $\cD^-$) the full subcategory of $\cD$ consisting of complexes $X$ such that $H^i(X)=0$ for $i\ll 0$ (resp. $i\gg 0$).
We write $\cD_{\rm fg}$ for the full subcategory of $\cD$ consisting of complexes with finitely generated cohomology modules.
Furthermore we write $\cD_{\rm fg}^{-}=\cD^-\cap \cD_{\rm fg}$.

For a complex $X$ in $\cD$, the (small) support of  $X$  is a subset of $\Spec R$ defined as 
$$
\supp X= \Set{ \fp \in \Spec R | X\Lotimes \kappa(\fp)\neq 0}, 
$$ 
where  
$\kappa(\fp)=R_\fp/\fp R_\fp$. 
It is well-known that for $X\in \cD$, $\supp X\neq \emptyset$ if and only if $X\neq 0$, see \cite[Lemma 2.6]{F} or \cite[Lemma 2.12]{N}. 
In order to compare with the ordinary support, recall that the (big) support  $\Supp X$  is the set of primes $\fp$  of $R$ satisfying  $X_{\fp} \neq 0$ in  $\cD$. 
In general, we have  $\supp X  \subseteq \Supp X$ and equality holds if $X \in \cD_{\rm fg}^-$, see \cite[p.\ 158]{F}.
 
A full subcategory $\cL$ of $\cD$ is said to be localizing  if $\cL$ is triangulated and  closed under arbitrary direct sums.
If we are given a subset $W$ of  $\Spec R$, then, since the tensor product commutes with taking direct sums, it is easy to see that 
the full subcategory  $\cL_W=\Set{X\in \cD | \supp X\subseteq W}$ is localizing.
A theorem of Neeman \cite[Theorem 2.8]{N} ensures that any localizing subcategory  of $\cD$ is obtained in this way from a subset $W$ of $\Spec R$.

If  $A$ is a set of objects in  $\cD$, then $\Loc A$ denotes the smallest localizing subcategory of  $\cD$  containing all objects of  $A$. 
We write $E_R(R/\fp)$ for the injective envelope of the $R$-module $R/\fp$  for  $\fp \in \Spec R$.
It is easy to see $\supp \kappa(\fp)=\supp E_R(R/\fp)=\{\fp\}$.
Moreover, Neeman \cite[Theorem 2.8]{N} shows the equalities 
$$
\cL_W =\Loc\Set{\kappa(\fp)|\fp\in W}
 =\Loc\Set{E_R(R/\fp)|\fp\in W}.
$$

For a localizing subcategory  $\cL$  of  $\cD$, its right orthogonal subcategory is defined as 
$$
\cL^\perp = \Set{Y \in \cD | \Hom_\cD(X,Y)=0\ \text{for all}\ X \in \cL}.
$$
Note that $\cL^\perp$ is a triangulated subcategory of $\cD$ that is closed under arbitrary direct products. 
In other words, $\cL ^\perp$ is a colocalizing subcategory of $\cD$.
The following equalities hold for any subset $W$ of $\Spec R$; 
$$
\begin{array}{ll}
\cL_W^\perp &= \Set{Y \in \cD | \Hom_\cD(\kappa(\fp)[i],Y)=0\ \text{for all}\ \fp \in W \ \text{and}\ i \in \Z} \vspace{3pt}\\
 &= \Set{Y \in \cD | \RHom_R(\kappa(\fp),Y)=0\ \text{for all}\ \fp \in W}. 
\end{array}
$$

Let  $W$ be an arbitrary subset of $\Spec R$. 
We denote by $i_W$ (resp. $j_W$) the natural inclusion functor $\cL_W\hookrightarrow \cD$ (resp. $\cL_W^\perp\hookrightarrow \cD$).
Then there exist a couple of adjoint pairs $(i_W, \gamma_W)$ and $(\lambda_W, j_W)$ as it is indicated in the following diagram:
$$
\xymatrix{
\cL_W\ \ar@<3.5pt>[rr]^*{i_W}&&\ar@<3.5pt>[ll]^*{\gamma_W}\ \cD\ \ar@<3.5pt>[rr]^*{\lambda_W}&&\ar@<3.5pt>[ll]^*{j_W}\ \cL_W^\perp}
$$
\color{red}
\color{black}
Moreover, it holds that $\Ker \gamma_W=\cL_W^\perp$ and $\Ker \lambda_W=\cL_W$.
See $\text{\cite[\S4.9, \S 5.1, \S 7.2]{K}}$ for details. See also Remark $\text{\ref{alternative proof}}$ (ii).

In the following lemma, we identify $\gamma_W$ and $\lambda_W$ with $i_W\cdot \gamma_W$ and $j_W\cdot \lambda_W$ respectively.

\begin{lemma} \label{LC-like}
Let  $W$  be any subset of  $\Spec R$. 
For any object $X$ of $\cD$, there is a triangle
$$
\begin{CD}
\gamma_W X @>>> X @>>> \lambda_W X @>>> \gamma _W X[1],
\end{CD}
$$
where $\gamma_W X\to X$ and $X\to \lambda_W X$ are the natural morphisms. 
Furthermore,
if 
$$
\begin{CD}
X' @>>> X @>>> X''@>>> X'[1]
\end{CD}
$$
is a triangle with $X' \in \cL_W$ and  $X'' \in \cL_W^{\perp}$, then there exist unique isomorphisms $a:\gamma _W X\to X'$ and  $b:\lambda _W X\to X''$ such that the following diagram is commutative:
$$
\begin{CD}
\gamma_W X @>>> X @>>> \lambda_W X @>>> \gamma _W X[1]\\
@VVaV  @| @VVbV@VVa[1]V \\
X' @>>> X @>>> X''@>>> X'[1]
\end{CD}
$$
\color{black}
\end{lemma}

See $\text{\cite[\S 4.11]{K}}$ for the proof of this lemma.\\
 
Let $W$ be a subset of $\Spec R$. Then $\lambda_W$ is  a localization functor.
In other words, writing $\eta:\id_\cD\to \lambda_W$ for the natural morphism, it follows that $\lambda_W\eta:\lambda_W\to \lambda_W^{2}$ is invertible and $\lambda_W\eta=\eta\lambda_W$.
Conversely, $\gamma_W$ is
a colocalization functor, that is, for the natural morphism
 $\varepsilon:\gamma_W\to \id_\cD$, it follows that $\gamma_W\varepsilon:\gamma_W^2\to \gamma_W$ is invertible and $\gamma_W\varepsilon=\varepsilon\gamma_W$. 
Note that we uniquely obtain the localization (resp. colocalization) functor $\lambda_W$ (resp. $\gamma_W$) for a subset $W$ of $\Spec R$.

\color{black}
 
\begin{definition}
Let  $W$  be a subset of  $\Spec R$.
We call $\gamma_W$ the colocalization functor with support in $W$.
\end{definition}

Recall that a subset $W$ of  $\Spec R$ is called specialization-closed (resp. generalization-closed) if 
 the following condition holds:

\vspace{6pt}\noindent 
(*)\ Let $\fp, \fq \in \Spec R$. 
If $\fp \in W$ and $\fp\subseteq \fq$ (resp. $\fp\supseteq \fq$), then 
$\fq$ belongs to $W$. 
\vspace{6pt}

If $V$ is a specialization-closed subset, then the colocalization functor $\gamma_V$ coincides with the right derived functor $\mathrm{R}\Gamma_V$ of the section functor $\Gamma_V$ with support in $V$, see \cite[Appendix 3.5]{L}.


\section{Auxiliary results on colocalization functors}
\label{3}
Let $W$ be a subset of $\Spec R$ and let  $\gamma _W$  be the colocalization functor with support in $W$. 
In general, it is hard to describe the functor $\gamma _W$ explicitly. 
However there are some cases in which the colocalization functor $\gamma _W$  is the form of composition of known functors. 

Let $S$ be a multiplicatively closed subset of $R$.
We denote by $U_S$ the generalization-closed subset $\Set{\fq\in \Spec R | \fq  \cap S=\emptyset }$. 
Note that $U_S$ is naturally identified with $\Spec S^{-1}R$.  
Following \cite{BIK}, 
we also write $U(\fp)=\Set{\fq\in \Spec R | \fq \subseteq \fp}$ for a prime ideal $\fp$ of $R$. Then, setting $S=R\backslash \fp$, we have $U(\fp)=U_S$.

\begin{proposition}\label{generalize gamma p}
Let $S$ be a multiplicatively closed subset of $R$ and $V$ be a specialization-closed subset of $\Spec R$. We set  
$W=V\cap U_S$.
Then we have an isomorphism
$$
\gamma_W \cong \RGamma_V\RHom_R(S^{-1}R,-).
$$
\end{proposition}

\begin{proof}
The ring homomorphism $R\rightarrow S^{-1}R$ induces a morphism $\RHom_R(S^{-1}R,X)\rightarrow X$ for $X \in \cD$.
Write $f:\RGamma_{V}\RHom_R(S^{-1}R,X)\rightarrow X$  for the composition of the natural morphism 
$\RGamma_{V}\RHom_R(S^{-1}R,X) \to  \RHom_R(S^{-1}R,X)$  with this morphism, and we consider 
the triangle
$$
\begin{CD}
\RGamma_{V}\RHom_R(S^{-1}R,X)@>{f}>> X@>>> C@>>> \RGamma_{V}\RHom_R(S^{-1}R,X)[1]
\end{CD}.
$$
Since the complex $\RGamma_{V} \RHom_R(S^{-1}R,X)$ can be regarded as a complex of  $S^{-1}R$-modules, 
it follows that $\supp \RGamma_{V}\RHom_R(S^{-1}R,X)\subseteq U_S$. 
At the same time,  it follows from the definition that $\supp \RGamma_{V}\RHom_R(S^{-1}R,X)\subseteq V$.
Therefore $\supp \RGamma_{V}\RHom_R(S^{-1}R,X)$ must be contained in $W$.

On the other hand, if $\fp\in W$, then there are isomorphisms 
\begin{align*}
\RHom_R(\kappa(\fp),\RGamma_{V}\RHom_R(S^{-1}R,X))&\cong \RHom_R(\kappa(\fp),\RHom_R(S^{-1}R,X))\\
&\cong \RHom_R(\kappa(\fp),X).
\end{align*}
It implies that $\RHom_R(\kappa(\fp),f)$  is an isomorphism, and thus  $\RHom_R(\kappa(\fp),C)=0$. 

Since we have shown that  $\RGamma_{V}\RHom_R(S^{-1}R,X) \in \cL_W$  and $C\in \cL_W^\perp$, 
we can use  Lemma \ref{LC-like} to deduce  $\gamma_WX\cong \RGamma_{V}\RHom_R(S^{-1}R,X)$. 
\end{proof}

In the following lemma we show that the colocalization functor considered in  Proposition \ref{generalize gamma p} is a right derived functor of a left exact functor defined on $\Mod R$.
We say that a complex $I$ of $R$-modules is $K$-injective if $\Hom_R(- ,I)$ preserves quasi-isomorphisms.

\begin{lemma}\label{gamma p 2}
Let $S, V$ and $W$ be the same as in Proposition \ref{generalize gamma p}. 
Then the colocalization functor $\gamma_{W}$ is the right derived functor of the functor $\Gamma_{V}\Hom_R(S^{-1}R,-)$  defined on $\Mod R$.
\end{lemma}

\begin{proof}
Let $X \in \cD$ and take a $K$-injective resolution $I$ of $X$ that consists of injective $R$-modules. 
Then $\RHom_R(S^{-1}R,X) \cong \Hom_R(S^{-1}R,I)$, and the right-hand complex consists of injective $R$-modules, too.
It is known by  \cite[Lemma 3.5.1]{L} that  for any complex $J$ of injective R-modules (that is not necessarily $K$-injective), $\RGamma_{V} J$ is naturally isomorphic to $\Gamma_{V}J$. 
Therefore we have 
$\RGamma_{V}\RHom_R(S^{-1}R,X)\cong \Gamma_{V}\Hom_R(S^{-1}R, I)\cong{\rm R}(\Gamma_{V}\Hom_R(S^{-1}R,-))(X)$.
\end{proof}

Henceforth, for a subset $W$ of $\Spec R$, we write $W^c=\Spec R\backslash W$. 
By Proposition $\text{\ref{generalize gamma p}}$, we have an isomorphism
$$
\gamma_{U_S}\cong \RHom_R(S^{-1}R,-).
$$
We should mention that the isomorphism $\gamma_{U(\fp)} \cong \RHom_R (R_\fp,-)$ already appeared in \cite[\S 4; P175]{BIK3}, in which the authors use the notation $V^{Z(\fp)}$, where $Z(\fp)=U(\fp)^c$, see also Remark \ref{RHom adjoint}.

\begin{corollary}\label{gamma p}
Let $\fp\in \Spec R$. 
Then we have an isomorphism  
$$
\gamma_{\{\fp\}} \cong \RGamma_{V(\fp)}\RHom_R(R_\fp,-),
$$
the right-hand side of  which is the right derived functor of  $\Gamma_{V(\fp)}\Hom_R(R_\fp,-)$ defined on $\Mod R$. 
\end{corollary}

By the corollary, if $I$ is an injective $R$-module, then 
 $\gamma_{\{\fp\}}I \cong \Gamma_{V(\fp)}\Hom_R(R_\fp,I)$ is also 
 an injective $R$-module. We can describe how this injective $R$-module is decomposed into a sum of indecomposable ones.

\begin{corollary}\label{gamma p 3}
Let $\fp$ be a prime ideal of $R$ and $I$ be an injective $R$-module. Then $\gamma_{\{\fp\}} I$ is isomorphic to the direct sum $\bigoplus_{B}E_R(R/\fp)$ of $B$-copies of $E _R(R/\fp)$, where $B=\dim_{\kappa(\fp)} \Hom_R(\kappa(\fp),I)$.
\end{corollary}

\begin{proof}
Since $\gamma_{\{\fp\}}I$ is an injective $R$-module with support in $\{ \fp\}$, there is a cardinal number $B$ with 
$\gamma_{\{\fp\}}I\cong \bigoplus_{B} E_R(R/\fp)$.
On the other hand, $\Hom_R(\kappa(\fp),I)\cong \Hom_R(\kappa(\fp),\gamma_{\{\fp\}}I)\cong \bigoplus_{B} \kappa(\fp)$, hence  $B=\dim_{\kappa(\fp)} \Hom_R(\kappa(\fp),I)$.
\end{proof}

\begin{remark}\label{non-zero gamma}
Let $I$ be an injective $R$-module such that $\supp I=\{\fq\}$ for $\fq\in \Spec R$, that is, $I$ is of the form $\bigoplus_{A}E_R(R/\fq)$ for some index set $A$.
Then it is easily seen that $\Hom_R(\kappa(\fp),I)\neq 0$ if and only if $\fp\subseteq \fq$.
Therefore, it follows from Corollary $\text{\ref{gamma p 3}}$ that 
$\gamma_{\{\fp\}}I\neq 0$ if and only if $\fp\subseteq \fq$.
\end{remark}

If $\fp$ is a prime ideal of $R$ which is not maximal, then 
the colocalization functor $\gamma_{\{\fp\}}$ is distinct from $\RGamma_{V(\fp)}((-)\otimes_RR_\fp) $, 
 which is written as  $\varGamma_\fp$ by Benson, Iyengar and Krause in \cite{BIK}.
In fact, for a prime ideal $\fq$ such that $\fp \subsetneq \fq$, it follows that $\varGamma _\fp E_R(R/\fq) = 0$, while $\gamma _{\{\fp\}} E_R(R/\fq) \not= 0$ by Reamrk \ref{non-zero gamma}.

\begin{definition}\label{def dimW}
For a subset $W$ of $\Spec R$,  we denote by $\dim W$ the supremum of the lengths of chains of prime ideals belonging to $W$, {\it i.e.}, 
$$
\dim W = \sup \Set{n | \text{there are} \ \fp_0, \ldots , \fp_n \ \text{in} \ W \ \text{with}\  \fp_0 \subsetneq \fp_1 \subsetneq \cdots \subsetneq \fp_n}.
$$
\end{definition}

Thus $\dim W=0$ means that two distinct prime ideals taken from $W$ have no inclusion relation. 
Moreover, if $W=\emptyset$, then $\dim W=-\infty$ by the definition.

We now want to extend Corollary \ref{gamma p} to the case where $\dim W=0$. 
For this purpose we need some preparatory observations. 
Compare the next remark with \cite[Lemma 3.4 (1)]{BIK}.

\begin{remark}\label{inclusion}
(i) Let $W_0$ and $W$ be subsets of $\Spec R$ with inclusion relation $W_0\subseteq W$. In this case, we should note that $\cL_{W_0}\subseteq \cL_{W}$ and $\cL_{W}^\perp \subseteq \cL_{W_0}^\perp$. 
Then it is clear from the uniqueness of adjoint functors that
$$
\gamma_{W_0}\gamma_{W} \cong  \gamma_{W_0} \cong  \gamma_{W}\gamma_{W_0}, \ \ \ 
\lambda_{W} \lambda_{W_0} \cong  \lambda_{W} \cong  \lambda_{W_0}\lambda_{W}.
$$
(ii) Let $W_1$ and $W_2$ be subsets of  $\Spec R$.
In general, $\gamma_{W_1} \gamma_{W_2}$ need not be isomorphic to $\gamma_{W_2} \gamma_{W_1}$.
For example, if $\fp, \fq\in \Spec R$ with $\fp\subsetneq \fq$, then it is seen from $\text{Corollary \ref{gamma p 3}}$ and Reamrk $\text{\ref{non-zero gamma}}$ that $\gamma_{\{\fp\}}\gamma_{\{\fq\}}E_R(R/\fq)\neq 0$ and $\gamma_{\{\fq\}}\gamma_{\{\fp\}}E_R(R/\fq)=0$.
Similarly, $\lambda_{W_1}\lambda_{W_2}$ need not be isomorphic to $\lambda_{W_2}\lambda_{W_1}$.
Moreover, for a general subset $W$, $\gamma_W$ dose not necessarily commute with the localization $(-)\otimes_RS^{-1}R$ with respect to a multiplicatively closed subset $S$.
\end{remark}\vspace{2mm}

The following lemma will be used in the later sections. 

\begin{lemma}[Foxby-Iyengar \cite{FI}]\label{Foxby-Iyengar}
Let $(R,\fm,k)$ be a commutative Noetherian local ring. 
Then the following conditions are equivalent for any  $X \in \cD$:
\begin{enumerate}[{\rm (1)}]
\item $X\Lotimes k  \neq 0$;
\item $\RHom_R(k,X) \neq 0$;
\item $\RGamma_{V(\fm)} X\neq 0$. 
\end{enumerate}
\end{lemma}
\begin{proof}
See \cite[Theorem 2.1, Theorem 4.1]{FI}.
\end{proof}

This is implicitly used by Benson, Iyengar and Krause \cite{BIK}
to prove the following lemma.

\begin{lemma}[{\cite[Theorem 5.6]{BIK}}]\label{BIK}
Let $V$ be a specialization-closed subset of $\Spec R$. Then, for each $X$ in $\cD$,  the following equalities hold;
$$
\supp \gamma_V X =V \cap \supp X, \quad  \supp \lambda_V X = V^c \cap \supp X.
$$
\end{lemma}\vspace{1mm}

Notice that Lemma \ref{BIK} applies only to specialization-closed subsets, and the equalities do not necessarily hold for general subsets, see Corollary \ref{gamma p 3} and Reamrk \ref{non-zero gamma}.

\begin{definition}\label{closure}
Let $W_{0} \subseteq W$  be subsets of  $\Spec R$.
We say that $W_{0}$ is {\it specialization-closed in} $W$ 
if $V(\fp) \cap W \subseteq W_{0}$ for any $\fp \in W_0$. 

Moreover we denote by $\overline{W}^s$ the specialization closure of $W$, which is defined to be the smallest specialization-closed subset of $\Spec R$ containing $W$. 
\end{definition}

\begin{lemma}\label{perp}
Let ${W_{0}} \subseteq W\subseteq \Spec R$ be sets. 
Suppose ${W_{0}}$ is specialization-closed in $W$.
Setting  ${W_{1}} = W \ \backslash {W_0}$, we have $\cL_{{W_{1}}}\subseteq \cL_{{W_0}}^\perp$.
\end{lemma}

\begin{proof}
It is obvious from the definition that 
 $\overline{{W_{0}}}^s \cap {W_{1}}= \emptyset$.
Assume that $X \in \cL_{{W_{1}}}$. 
Then $\supp \gamma_{\overline{{W_{0}}}^s} X = \overline{{W_{0}}}^s \cap \supp X = \emptyset $ by Lemma \ref{BIK}. 
Therefore we have  $\gamma_{\overline{{W_{0}}}^s}X=0$.
It then follows from Remark \ref{inclusion} (i) that $\gamma_{{W_{0}}} X \cong \gamma_{{W_{0}}} \gamma_{\overline{{W_{0}}}^s}X = 0$.
Thus we have $X \cong \lambda _{{W_{0}}}X  \in \cL_{{W_{0}}}^\perp$ as desired.
\end{proof}

The following theorem is one of the main results in this section; it extends Corollary \ref{gamma p}.

\begin{theorem}\label{dimW=0 2}
Let $W$ be a subset of $\Spec R$ with $\dim W=0$. 
Then we have the following isomorphisms of functors
$$
\gamma_W \cong \bigoplus_{\fp \in W}\gamma_{\{\fp\}}
\cong \bigoplus_{\fp \in W} \RGamma_{V(\fp)}\RHom_R(R_\fp,-).
$$
Furthermore, $\gamma_W$ is the right derived functor of the left exact functor 
$$\disp{\bigoplus_{\fp \in W}\Gamma_{V(\fp)}\Hom_R(R_\fp,-)}
$$
defined on $\Mod R$.
\end{theorem}

\begin{proof}
Let $X\in \cD$. 
Summing up all the natural morphisms $\gamma_{\{\fp\}} X\rightarrow X$ for  $\fp\in W$, 
we obtain a morphism $f:\bigoplus_{\fp \in W} \gamma_{\{\fp\}}X\rightarrow X$, 
from which we obtain a triangle
$$
\begin{CD}
{\displaystyle \bigoplus_{\fp \in W}} \gamma_{\{\fp\}}X @>{f}>>  X 
@>>> C @>>> {\displaystyle \bigoplus_{\fp \in W}} \gamma_{\{\fp\}}X[1].
\end{CD}
$$
It is clear that $\bigoplus_{\fp \in W} \gamma_{\{\fp\}}X\in \cL_W$. 

Now let  $\fp$  be a prime in $W$. 
Since $\{\fp\}$ is specialization-closed in $W$, it follows from Lemma \ref{perp} that 
$\bigoplus_{\fq \in W\backslash \{\fp\}}\gamma_{\{\fq\}}X\in \cL_{\{\fp\}}^\perp$.
Hence we have
$$
\RHom_R(\kappa(\fp),\bigoplus_{\fq \in W} \gamma_{\{\fq\}}X)\cong \RHom_R(\kappa(\fp), \gamma_{\{\fp\}}X)\cong \RHom_R(\kappa(\fp), X).
$$
This implies that $\RHom_R(\kappa(\fp),f)$ is an isomorphism for $\fp\in W$. 
Therefore it follows that $\RHom_R(\kappa(\fp), C)=0$ for all $\fp\in W$, equivalently $C\in \cL_{W}^\perp$.
Hence we conclude by Lemma \ref{LC-like} that $\gamma_WX\cong \bigoplus_{\fp \in W}\gamma_{\{\fp\}}X$ as desired.
The rest of the theorem follows from Lemma \ref{gamma p 2} and Corollary \ref{gamma p}. 
\end{proof}\vspace{2mm}

We denote by $\Inj R$ the full subcategory of $\Mod R$ consisting of all injective $R$-modules.
When $\dim W>0$, even for $I\in \Inj R$, it may happen that there is a negative integer $i$ with  $H^i(\gamma_WI)\neq 0$,
see Example \ref{negative cohomology}. Therefore, for a general subset $W$ of $\Spec R$, $\gamma_W$ is not necessarily a right derived functor of an additive functors defined on $\Mod R$.

The following theorem enables us to compute $\gamma _W X$ by using induction on $\dim W$.

\begin{theorem}\label{theorem1}
 Let ${W_{0}} \subseteq  W \subseteq \Spec R$  be sets. 
 Assume that ${W_{0}}$ is specialization-closed in $W$, and set ${W_{1}}=W \ \backslash \ {W_{0}}$.
For any $X\in \cD$, there is a triangle of  the following form;
$$
\begin{CD}
\gamma_{{W_{0}}} X @>>>  \gamma_W X @>>> \gamma_{{W_{1}}} \lambda_{{W_{0}}} X @>>> 
\gamma_{{W_{0}}} X[1].
\end{CD}
$$
\end{theorem}

\begin{proof} 
By virtue of Lemma \ref{LC-like}, we have triangles;
$$\begin{CD}
\gamma_{{W_{0}}} X @>>>  X @>>> \lambda_{{W_{0}}} X @>>>  \gamma_{{W_{0}}} X [1],
\end{CD}
$$
$$
\begin{CD}
\gamma_{{W_{1}}} \lambda_{{W_{0}}} X @>>> \lambda_{{W_{0}}}X @>>> \lambda_{{W_{1}}} \lambda_{{W_{0}}} X @>>> 
\gamma_{{W_{1}}} \lambda_{{W_{0}}} X[1].
\end{CD}
$$
It then follows from the octahedron axiom that there is a commutative diagram whose rows and columns are triangles: 
$$
\begin{CD}
@. X[-1]  @=  X[-1]    @. \\
@.     @VVV    @VVV     @. \\
\gamma_{{W_{1}}}\lambda_{{W_{0}}}X[-1] @>>> \lambda_{{W_{0}}} X [-1] @>>> \lambda_{{W_{1}}}\lambda_{{W_{0}}} X [-1] @>>>  \gamma_{{W_{1}}}\lambda_{{W_{0}}}X \\
@|     @VVV     @VVV     @| \\
\gamma_{{W_{1}}}\lambda_{{W_{0}}}X[-1] @>>> \gamma_{{W_{0}}} X @>>>  C @>>> \gamma_{{W_{1}}}\lambda_{{W_{0}}}X \\
@.     @VVV    @VVV     @. \\
@. X  @=  X    @. \\
\end{CD}
$$

Focusing on the triangle in the second row,  
we notice from Lemma \ref{perp} that both $\gamma_{{W_{1}}}\lambda_{{W_{0}}}X$ and  $\lambda _{{W_{0}}} X$  belong to  $\cL_{{W_{0}}}^{\perp}$.
Hence we have  $\lambda_{{W_{1}}}\lambda_{{W_{0}}}X\in \cL_{{W_{0}}}^\perp$. 
Then it follows that 
$\lambda_{{W_{1}}} \lambda_{{W_{0}}} X   \in  \cL_{{W_{1}}}^\perp  \cap \cL_{{W_{0}}}^{\perp} = \cL_{W}^\perp$.

On the other hand, in the third row above,  
we know  $\gamma_{{W_{1}}} \lambda_{{W_{0}}} X \in \cL_{{W_{1}}} \subseteq \cL_W$  and 
 $\gamma_{{W_{0}}} X \in \cL_{{W_{0}}} \subseteq \cL_W$. 
Consequently we have  $C \in \cL_W$. 
 
Taking a look at the third column above,  
since  $C \in  \cL_W$ and $\lambda_{{W_{1}}}\lambda_{{W_{0}}} X \in  \cL_{W}^\perp$, 
we deduce from Lemma \ref{LC-like}  that $\gamma_W X \cong C$.
Thus the third row is a required triangle. 
\end{proof}

The reader should compare Theorem \ref{theorem1} with \cite[Lemma 3.4 (4)]{BIK}.

\begin{remark}\label{alternative proof}
(i) Let $W$, $W_0$ and $W_1$ be as in the theorem. In its proof, we have shown an isomorphism  $\lambda_{{W_{1}}}\lambda_{{W_{0}}}\cong \lambda_W$.

(ii) Let $W$ be a subset of $\Spec R$, and assume that $\dim W=n$ is finite. 
Then we can give an alternative proof of the existence of $\gamma_W$ and $\lambda_W$.
In fact, in the case that $n=0$, $\gamma_W$ can be given explicitly by Theorem \ref{dimW=0 2}. 
Note that $\lambda_W$ exists at the same time.
Furthermore, if $n>0$, we can show the existence of $\gamma_W$ and $\lambda_W$ inductively by the formula in (i).
\end{remark}

Let $X$ be a complex of $R$-modules.
We say that $X$ is left (resp. right) bounded if $X^i=0$ for $i\ll 0$ (resp. $i\gg 0$). When $X$ is left and rignt bounded, $X$ is called bounded.
We denote by $\cK=K(\Inj R)$  the homotopy category of complexes of injective $R$-modules.
Moreover, we write $\cK^+$ for the full subcategory of $\cK$ consisting of left bounded complexes.
Let $a$ and $b$ be taken from $\Z \cup \{+\infty\}$, and assume that 
$a \leq b$. 
We denote by $\cK ^{[a,b]}$ the full subcategory of $\cK^+$ consisting of complexes $I$ such that $I^i=0$ for $i\notin [a,b]$ (cf. \cite[Notations 11.3.7 (ii)]{KS}).

Recall that the canonical functor $\cK\to \cD$ induces an equivalence $\cK^+\xrightarrow{\sim} \cD^+$ of triangulated categories, whose quasi-inverse sends a complex $X\in \cD^+$ to its minimal injective resolution $I\in \cK^+$.
In the following corollary, we identify $\cK^+$ with $\cD^+$ in this way.

\begin{corollary}\label{[a,b]}
Let $W$ be a subset of $\Spec R$, and 
assume that $n = \dim W$  is finite. 
Let $a,b \in \mathbb{Z} \cup \{+\infty \}$ with $a\leq b$ and $I\in \cK^{[a,b]}$. 
Then $\gamma_WI$ is belongs to $\cK^{[a-n,b]}$ under the equivalence $\cK^+\cong\cD^+$.
Therefore, $\gamma _W$ maps objects of $\cD^+$ to objects of $\cD^+$.
\end{corollary}

\begin{proof}
We prove the corollary by induction on $n$.
If $n=0$, then it follows from Theorem \ref{dimW=0 2} that $\gamma_WI\in \cK^{[a,b]}$. 

Suppose that  $n>0$. 
Let ${W_{0}}$ be the set of all prime ideals in  $W$ that are maximal with respect to the inclusion relation in $W$, and we set ${W_{1}} = W  \ \backslash \ {W_{0}}$.
Notice that  $\dim {W_{0}} = 0$ and $\dim {W_{1}} = n-1$.
Since there is a triangle  $\gamma _{{W_{0}}} I  \to  I  \to \lambda _{{W_{0}}} I \to \gamma _{{W_{0}}} I [1]$, and since 
both $I$ and  $\gamma _{{W_{0}}} I$ belong to $\cK ^{[a, b]}$, we see that 
$\lambda _{{W_{0}}} I \in \cK^{[a-1, b]}$. 
Hence  $\gamma _{{W_{1}}} \lambda _{{W_{0}}} I \in \cK^{[a-n, b]}$ by the inductive hypothesis. 
  
On the other hand, we have from Theorem \ref{theorem1} a triangle
$$
\gamma_{{W_{0}}} I \to \gamma_W I \to \gamma_{{W_{1}}}\lambda_{{W_{0}}} I \to \gamma _{{W_{0}}} I[1].
$$
Since $\gamma_{{W_{0}}} I \in \cK^{[a,b]}$  by Theorem \ref{dimW=0 2}, and  since  $\gamma_{{W_{1}}}\lambda_{{W_{0}}} I \in \cK^{[a-n,b]}$ as shown in above, it follows that $\gamma_W I \in \cK^{[a-n,b]}$ as desired. 
\end{proof}

If $\dim W$ is infinite, then it may happen that $\gamma_WX\notin \cD^+$ for a complex $X\in \cD^+$, see Example \ref{dim W infinite}.


\section{Local Duality Principle}
\label{4}
Local duality theorem is a duality concerning local cohomology modules with supports in closed subsets in schemes, which was presented in \cite{Ha} and \cite{Ha2}.  
Dualizing complexes or dualizing modules play a significant role there. 
However, Foxby $\text{\cite[Proposition 6.1]{F}}$ discovered a general principle that underlies local duality, which does not require dualizing complexes.
He considered such duality only for the right derived functor $\RGamma_V$ of the section functor $\Gamma_{V}$ with support in a specialization-closed subset $V$ of $\Spec R$. 
We propose the local duality principle as generalization of Foxby's theorem.

\begin{theorem}[Local Duality Principle]\label{LD Principle}
Let $W$ be a subset of $\Spec R$ and let  $X, Y \in \cD$.
Suppose that one of the following conditions holds: 
\begin{enumerate}
\item[$(1)$] $X\in \cD^-_{\rm fg}$, $Y \in \cD^+$ and $\dim W <+\infty$;
\item[$(2)$] $X\in \cD_{\rm fg}$, $Y$ is a bounded complex of injective $R$-modules and $\dim W<+\infty$;
\item[$(3)$] $W$ is generalization-closed. 
\end{enumerate}
Then there exist natural isomorphisms
$$
\gamma_W \RHom_R(X, Y)\cong \RHom_R(X,\gamma_W Y),$$
$$\lambda_W\RHom_R(X,Y)\cong \RHom_R(X,\lambda_WY).$$
\end{theorem}\vspace{2mm}

Note that Foxby's theorem states the validity of the first isomorphism when, 
added to the condition (1), $W$  is a specialization-closed subset of $\Spec R$.

\begin{lemma}\label{LWperp}
Let $W$ be a subset of $\Spec R$.
If  $Z\in \cL_W^\perp$, then $\RHom _R (X,Z)\in \cL_W^\perp$ for any $X \in \cD$.
\end{lemma}

\begin{proof}
The lemma is clear from 
$$
\RHom _R(Y, \RHom_R (X,Z)) \cong  \RHom_R (X, \RHom_R (Y, Z)) = 0
$$
for  $Y \in \cL_W$. 
\end{proof}

\begin{lemma}\label{localizing and colocalizing}
Let $W$ be a generalization-closed subset of $\Spec R$.
Then it holds that
$$
\cL_W=\cL_{W^c}^\perp. 
$$
\end{lemma}

\begin{proof}
We note that $W^c$ is specialization-closed.
The equality $
\cL_W=\cL_{W^c}^\perp$ is deduced from Lemma \ref{BIK} as follows:
If $X \in \cL_W$, then 
$\supp \gamma_{W^c}X = W^c\cap \supp X = \emptyset$ hence $\gamma _{W^c}X=0$ equivalently  $X = \lambda _{W^c}X \in \cL_{W^c}^\perp$.
On the contrary, if  $X \in \cL_{W^c}^\perp$ then 
$\supp \lambda _{W^c} X = W\cap \supp X$ therefore $X = \lambda _{W^c}X$ has small support in $W$ thus  $X \in \cL_W$. 
\end{proof}

Let $X$ be a complex of $R$-modules and  $n$ be an integer.
The cohomological truncations $\sigma_{\leq n}X$ and $\sigma_{> n}X$ are defined as follows: 
\begin{align*}\sigma_{\leq n}X&=(\cdots \to X^{n-2}\to X^{n-1}\to \Ker d^{n}_X\to 0\to\cdots )\\
\sigma_{> n}X&=(\cdots\to 0\to \Image d_X^{n}\to X^{n+1}\to X^{n+2}\to\cdots )
\end{align*}
See \cite[Chapter I; \S 7]{Ha} for details.

\begin{proof}[Proof of Theorem \ref{LD Principle}]
Applying the functor  $\RHom _R (X,-)$ to the triangle 
$\gamma_W Y \to Y \to \lambda_W Y \to \gamma _W Y [1]$, 
we obtain a triangle of the form; 
$$
\RHom _R (X, \gamma_W Y) \to \RHom _R (X,Y) \to \RHom_R (X, \lambda_W Y) \to \RHom _R (X, \gamma_W Y) [1] .
$$
It follows from  Lemma \ref{LC-like} that, to prove the desired isomorphisms, 
we only have to show that 
$$
\RHom_R(X,\gamma_WY)\in \cL_W
 \ \ \text{and}  \ \ 
\RHom _R (X,\lambda_W Y) \in \cL_W^\perp. 
$$
Since  $\lambda _W Y \in \cL_W ^{\perp}$, we see from Lemma \ref{LWperp} that  $\RHom_R(X,\lambda_WY)\in \cL_W^\perp$.
Thus it remains to show that $\RHom_R(X,\gamma_WY)\in \cL_W$. 

Case (1): Let $\fp  \in W^c$.
We want to show that  $\RHom_R (X, \gamma_W Y) \Lotimes  \kappa ( \fp ) = 0$. 
Since $X \in \cD^-_{\rm fg}$ and $Y \in \cD^+$,  Corollary $\text{\ref{[a,b]}}$ implies $\gamma_WY\in \cD^+$,
so it follows that
$$
 \RHom_R (X, \gamma_W Y) \Lotimes \kappa(\fp) \cong \RHom_{R_\fp} ( X_\fp ,(\gamma_WY)_\fp) \otimes^{\rm L}_{R_\fp}\kappa(\fp).
 $$ 
Now thanks to Lemma \ref{Foxby-Iyengar} together with this isomorphism, it is sufficient to show that 
$\RHom_{R_\fp} \left(\kappa(\fp), \RHom_{R_\fp} (X_\fp, (\gamma_WY)_\fp)\right)=0$. 
Noting that 
$$\RHom_{R_\fp}\left(\kappa(\fp), \RHom_{R_\fp}(X_\fp,(\gamma_WY)_\fp)\right)\cong 
\RHom_{R_\fp}\left(X_\fp, \RHom_{R_\fp}(\kappa(\fp),(\gamma_WY)_\fp)\right),
$$ 
we have sufficiently to show that $\RHom _{R_\fp} ( \kappa(\fp), ( \gamma_W Y )_\fp ) =0$. 
However, by using Lemma \ref{Foxby-Iyengar} again, we see that this is equivalent to show that  $(\gamma_W Y) \otimes^{\rm L}_R \kappa(\fp)=0$, {\it i.e.}, $\fp \not\in  \supp \gamma  _W Y$.
The last is clear, since  $\supp \gamma _W Y \subseteq W$ and $\fp \not\in W$.

Case (2): As in the case (1), taking $\fp\in W^c$, we show $\RHom_R (X, \gamma_W Y) \Lotimes  \kappa ( \fp ) = 0$.
We consider the triangle
$\sigma_{\leq n}X\to X\to \sigma_{> n}X\to (\sigma_{\leq n}X)[1]$ for an integer $n$.
Since $\sigma_{\leq n}X\in \cD^{-}_{\rm fg}$, we have $ \RHom_R(\sigma_{\leq n}X,\gamma_WY)\in \cL_W$ by the case (1).
Hence, applying $\RHom_R(-,\gamma_WY)\Lotimes\kappa(\fp)$ to the triangle, 
we obtain an isomorphism
$$\RHom_R(X,\gamma_WY)\Lotimes\kappa(\fp)\cong \RHom_R(\sigma_{>n}X,\gamma_WY)\Lotimes\kappa(\fp).$$
Let $i$ be any integer.  It suffices to show that $$H^0\left(\RHom_R(\sigma_{>n}X,\gamma_WY[i])\Lotimes\kappa(\fp)\right)=0$$
for some $n$.
By Corollary \ref{[a,b]}, $\gamma_WY$ is isomorphic to a bounded complex $I$ of injective $R$-modules, so that there is an integer $m$ with $I^j=0$ for $j>m$.
Moreover, any element of $H^0(\RHom_R(\sigma_{>n}X,\gamma_WY[i]))\cong \Hom_{\cD}(\sigma_{>n}X,I[i])$ is represented by a chain map $\sigma_{>n}X\to I[i]$.
Thus, taking $n$ with $n>m-i$, we see that $H^j(\RHom_R(\sigma_{>n}X,\gamma_WY[i]))\cong\Hom_{\cD}(\sigma_{>n}X,I[i+j])\big)=0$ for all $j\geq 0$. 
Then it is easily seen that 
$H^0\left(\RHom_R(\sigma_{>n}X,\gamma_WY[i])\Lotimes\kappa(\fp)\right)=0$.

Case (3): By Lemma \ref{localizing and colocalizing}, we have $\gamma _W Y \in \cL_W = \cL_{W^c}^{\perp}$, thus it follows from Lemma \ref{LWperp} that 
$\RHom _R (X, \gamma _W Y) \in  \cL_{W^c}^{\perp} = \cL_W$  as desired. 
\end{proof}

\begin{remark}\label{LDP remark}
In the case (3), the isomorphisms in the theorem are also proved by \cite[Theorem 5.14]{AJS}.
\end{remark}

When $R$ admits a dualizing complex $D_R$,  we write $X^\dagger=\RHom_R(X,D_R)$ for $X\in \cD$. Then we have $X\cong X^{\dagger\dagger}$ for $X\in \cD_{\rm fg}$, see 
\cite[Chapter V; \S 2]{Ha}. The following result is the generalized form of the local duality theorem \cite[Chapter V; Theorem 6.2]{Ha}. 

\begin{corollary}
\label{LDT}
Assume that $R$ admits a dualizing complex $D_R$.
Let $W$ be a subset of $\Spec R$ and $X\in \cD_{\rm fg}$.
Then we have a natural isomorphism 
$$\gamma_WX\cong \RHom_R(X^\dagger, \gamma_WD_R).$$
\end{corollary}
\vspace{1mm}

The second author and Maiko Ono had proved the following theorem when $W$ is specialization-closed, and had asked if it holds for an arbitrary subset $W$, see \cite[Theorem 2.3, Question 2.5]{OY}.

\color{black}
\begin{theorem}[AR Principle]\label{ARprinciple}
Let $X,I \in \cD$. 
Suppose that a subset  $W$ of $\Spec R$ satisfies the condition $\dim W <  +\infty$. 
We assume furthermore that the following conditions hold for an integer $n$:
\begin{enumerate}[{\rm (1)}]
\item $I$ is a bounded complex of injective $R$-modules with $I^i=0$ for $i>n$;
\item $\sigma_{\leq -1} X \in \cL_W$.
\end{enumerate}
Then there exists a natural isomorphism
$$
\sigma_{> n}\RHom_R(X,\gamma_W I)\cong \sigma_{> n}\RHom_R(X,I).
$$
\end{theorem}\vspace{2mm}

We are now able to prove that this theorem holds in general.
As we have shown in Corollary \ref{[a,b]},  $\gamma _W I$ is isomorphic to a bounded complex $J$ of injective $R$-modules with $J^i=0$ for $i>n$.
Then one can observe that the proof of \cite[Theorem 2.3]{OY} works well. 

The AR Principle is a version of classical Auslander-Reiten duality theorem in terms of complexes, see \cite[Corollary 3.2]{OY}.


\section{Relation with completion}
\label{5}

In this section, we shall explain the relationship between $\lambda_W$ and left derived functors of completion functors. Furthermore, we give some nontrivial examples of colocalization functors $\gamma_W$, for which $\gamma_W I$ has a non-zero negative cohomology module even for an injective $R$-module $I$. 

Let $\fa$ be an ideal of $R$ which defines a closed subset  $V(\fa)$ of $\Spec R$. 
We denote by $\Lambda^{V(\fa)}$ the $\fa$-adic completion functor $\varprojlim (-\otimes _RR/\fa ^n)$ defined on $\Mod R$. 
It is known that the left derived functor $\LLambda ^{V(\fa)}$ of $\Lambda^{V(\fa)}$ is a right adjoint to $\RGamma _{V(\fa)}$ by Greenlees and May \cite{GM} and  Alonso Tarr\' io, Jerem\' ias L\' opez and  Lipman \cite{AJL}.
\color{black}  
One also finds an outline of the proof of this fact in \cite[\S 4; p.\ 69]{L}. 

Recall that $\lambda_{V(\fa)^c}$ is a left adjoint to the inclusion functor $j_{V(\fa)^c}:\cL_{V(\fa)^c}\hookrightarrow \cD$.
Now we shall prove that $\LLambda ^{V(\fa)}$ coincides with $\lambda_{V(\fa)^c}$.
By the universality of derived functors, there is a natural morphism $X \to \LLambda ^{V(\fa)} X$ for any $X \in \cD$
, from which we have a triangle of the form
$$\begin{CD}
C @>>> X @>>> \LLambda ^{V(\fa)}X @>>> C [1]. 
\end{CD}$$
Applying $\RGamma_{V(\fa)}$ to this triangle,
we have the following triangle
$$\begin{CD}
\RGamma_{V(\fa)} C @>>> \RGamma_{V(\fa)} X @>>> \RGamma_{V(\fa)} \LLambda ^{V(\fa)}X @>>>\RGamma_{V(\fa)} C [1]. 
\end{CD}$$
By \cite[Corollary 5.1.1 (ii)]{AJL}, 
$\RGamma_{V(\fa)} X \rightarrow \RGamma_{V(\fa)} \LLambda ^{V(\fa)}X$ is an isomorphism.
Hence 
$\RGamma_{V(\fa)} C=0$, and thus 
we have $C\in \cL_{V(\fa)}^\perp=\cL_{V(\fa)^c}$ by Lemma \ref{localizing and colocalizing}.
On the other hand, since $\LLambda^{V(\fa)}$ is a right adjoint to $\RGamma_{V(\fa)}$, we have $$\RHom_R(\kappa(\fp), \LLambda ^{V(\fa)}X)\cong \RHom_R(\RGamma_{V(\fa)}\kappa(\fp), X)=0$$ for any $\fp \in V(\fa)^c$.
This implies $\LLambda ^{V(\fa)}X\in \cL_{V(\fa)^c}^\perp$.
Thus it follows that $\gamma_{V(\fa)^c} X\cong C$ and $\lambda_{V(\fa)^c}X\cong \LLambda^{V(\fa)} X$ by Lemma \ref{LC-like}. 

We summarize this fact in the following. 

\begin{proposition}\label{left derived completion}
Let $\fa$ be an ideal of $R$.
Then $\lambda_{V(\fa)^c}$ is isomorphic to the left derived functor $\LLambda^{V(\fa)}$ of the $\fa$-adic completion functor $\Lambda^{V(\fa)}$.
\end{proposition}

\begin{remark}\label{RHom adjoint}
Let $\fa$ be an ideal of $R$ and $W$ be  a specialization-closed subset of $\Spec R$.
Note that it is also proved that $\LLambda^{V(\fa)}$ is isomorphic to $\RHom_R(\RGamma_{V(\fa)} R,-)$ in \cite{GM} and \cite{AJL}. 
More generally, we see that $\RHom_R(\RGamma_{W} R,-)$ is a right adjoint to 
$\RGamma_W$. 
Furthermore, by using Lemma \ref{localizing and colocalizing}, it is not hard to see that $\lambda_{W^c}$ is a right adjoint to $\gamma_W$.
Thus it follows from the uniqueness of adjoint functors  that there is an isomorphism 
$$\lambda_{W^c}\cong \RHom_R(\RGamma_{W} R,-).$$
This fact and Proposition \ref{left derived completion} is essentially stated in \cite{BIK3}, where
$\gamma_{W^c}$ and $\lambda_{W^c}$ appear as $V^{W}$ and $\Lambda^{W}$ respectively. 
\end{remark}

Now we are ready to give an example as we have previously announced. 
 
\begin{example}\label{negative cohomology}
Let $(R,\fm,k)$ be a local ring of dimension $d$ and $\widehat{R}$ be the $\fm$-adic completion of $R$.
Let $D_{\widehat{R}}$ be a dualizing complex of $\widehat{R}$ with $\Ext^d_{\widehat{R}}(k,D_{\widehat{R}})\cong k$. 
We regard $D_{\widehat{R}}$ as a complex of $R$-modules in a natural way.
Using the isomorphism $\LLambda^{V(\fm)}\RGamma_{V(\fm)}\cong \LLambda^{V(\fm)}$ by $\text{\cite[Corollary 5.1.1 (i)]{AJL}}$, we can show that $\LLambda^{V(\fm)}E_R(k)\cong D_{\widehat{R}}[d]$.
Hence it follows from Proposition $\text{\ref{left derived completion}}$ that $\lambda_{V(\fm)^c}E_R(k)\cong D_{\widehat{R}}[d]$.
Now we suppose that $d>1$.
Then, considering the triangle
$$
\begin{CD}
\gamma_{V(\fm)^c}E_R(k)  @>>>  E_R(k)  @>>>  \lambda_{V(\fm)^c}E_R(k) @>>> \gamma_{V(\fm)^c}E_R(k)[1],
\end{CD}
$$
we have $H^{-d+1}(\gamma_{V(\fm)^c} E_R(k))\neq 0$ and $-d+1<0$.
\qed
 \end{example}

\begin{remark}\label{cohomology remark}
Let $X\in \cD$.
If $W$ is specialization-closed, then it holds that $\supp H^i(\gamma_W X)\subseteq W$ for all $i\in \mathbb{Z}$, since $\gamma_W\cong \RGamma_W$.
However we see from Example \ref{negative cohomology} that the inclusion relation $\supp H^i(\gamma_WX)\subseteq W$ dose not necessarily hold in general. 
\end{remark}

\color{black}

We now give an example such that $\gamma_WI\notin \cD^+$ even for an injective $R$-module $I$.

\begin{example}\label{dim W infinite}
We assume that $\dim R=+\infty$. 
Let $W$ be the set of maximal ideals of $R$.
Then we can show that $\gamma_{W^c}(\bigoplus_{\fm\in W}E_R(R/\fm))\notin \cD^+$. 
In fact,  it is clear that $\RGamma_W\cong \bigoplus_{\fm\in W}\RGamma_{V(\fm)}$. 
Thus it follows from Remark \ref{RHom adjoint} that 
$$\lambda_{W^c}\cong \RHom_R(\RGamma_WR,-)\cong \prod_{\fm\in W}\RHom_R(\RGamma_{V(\fm)}R,-)\cong \prod_{\fm\in W}\LLambda^{V(\fm)}.$$
Then, by Example \ref{negative cohomology}, we see that  $\lambda_{W^c}(\bigoplus_{\fm\in W}E_R(R/\fm))\notin\cD^+$. Hence we have $\gamma_{W^c}(\bigoplus_{\fm\in W}E_R(R/\fm))\notin\cD^+$.
\qed
\end{example}

\begin{remark}\label{subsequent}
More generally, we can prove that $\lambda_{W^c}$ is isomorphic to 
$$\prod_{\fp\in W}\LLambda^{V(\fp)}(-\otimes_R R_\fp)$$
 for an arbitrary  subset $W$ of $\Spec R$ with $\dim W=0$. The proof of this fact will be given in our subsequent work \cite{NY}.
\end{remark}


\section{Vanishing for cohomology modules}
\label{6}

Let  $V$  be a specialization-closed subset of $\Spec R$. 
In such a classical case, Grothendieck showed that if  $M$ is a finitely generated $R$-module, then $H^i_V (M)=H^i(\RGamma_VM)=0$ for $i<0$ and $i>\dim M$.

Our aim in this section is to prove the same type of vanishing theorem holds  for the colocalization functor $\gamma_W$ with support in $W$,  where $W$ is not necessarily specialization-closed.

Notice from Example \ref{negative cohomology} that it is truly nontrivial even to prove that $H^i(\gamma_W M)=0$ for $i<0$.

\begin{proposition}\label{negative vanishing}
Let $W$ be a subset of  $\Spec R$ and suppose that $\dim W $ is finite. 
Then we have $H^i (\gamma_WM) =0$ for any $i<0$ and for any finitely generated $R$-module $M$.
\end{proposition}

\begin{proof}
First of all we note the following: 
Let   $0 \to N' \to N \to N'' \to 0$  be an exact sequence of finitely generated $R$-modules. 
If  $N$ is a counterexample to the theorem, then so is one of  $N'$  and $N''$.
   
Secondly we note that  a finitely generated $R$-module $M$ has a filtration
$$
0=M_0\subsetneq M_1\subsetneq \cdots \subsetneq M_n=M
$$
such that $M_{i+1}/M_i\cong R/\fp_i$  and  $\fp_i \in \Spec R$ for $0 \leq i < n$. 
It thus follows that we have sufficiently to prove that $H^i(\gamma_WR/\fp)=0$  for $i < 0$ and $\fp\in \Spec R$.

Now supposed  that there would exist $\fp \in \Spec R$ with $H^i(\gamma_W R/ \fp) \neq 0$ for some $i<0$, and take a maximal   $\fp$  among such prime ideals.
Then it is easy to see from the remark made in the first part of this proof that the theorem is true for  $M = R/I$ for all $I\supsetneq \fp$. 
 
The key for the proof is Corollary \ref{[a,b]} which states that  $\gamma_W R/ \fp \in \cD^+$. 
This exactly means there is the least integer $\ell < 0$ with  $H^\ell(\gamma_W R / \fp ) \neq 0$.
Let $x$  be an arbitrary element of $R$ with  $x \not\in \fp$. 
Apply the functor $\gamma_W$ to the exact sequence
$$0 \to R/  \fp  \xrightarrow{x} R/ \fp \to R / (\fp+(x)) \to 0,$$
and we obtain an isomorphism  $H^{\ell}(\gamma_W R / \fp ) \xrightarrow{x} H^{\ell} (\gamma_W R / \fp)$.
Since every element from  $R \ \backslash \ \fp$ acts bijectively on  $H^\ell(\gamma_WR/ \fp)$,  consequently $H^\ell(\gamma_WR / \fp)$ is a $\kappa(\fp)$-vector space.
Now let  $I$  be a minimal injective resolution of $\gamma _W R/ \fp$. 
By our choice of integer  $\ell$, we see that $I$ starts from the $\ell$th term, being of the form 
$$
\begin{CD} 
0 @>>> I^{\ell}  @>{\partial^{\ell}}>> I^{\ell+1} @>{\partial^{\ell+1}}>> I^{\ell+2} @>{\partial^{\ell+2}}>> \cdots.
\end{CD}
$$
Note that  $\Ker  \partial^{\ell} = H^\ell(\gamma_WR/\fp)$. 
Since  $I$ is a minimal injective resolution, $I^{\ell}$  is an essential extension of $H^{\ell} (\gamma_WR/ \fp)$ that is a $\kappa (\fp)$-module. 
Therefore $I^{\ell}$ must be isomorphic to a direct sum of copies of $E_R(R / \fp)$.
This forces that  $\Gamma _{V(\fp)}(I _{\fp})$ has a nontrivial cohomology in degree $\ell$, thus it follows from 
 Lemma \ref{Foxby-Iyengar} that $\fp \in \supp I = \supp \gamma _W R/\fp \subseteq W$. 
 Since we have shown that $\fp \in W$, the following isomorphisms hold;  
$$
0 \neq \Hom_\cD(\kappa(\fp), \gamma_WR/\fp [\ell]) \cong \Hom_\cD (\kappa(\fp), R/\fp[\ell]) \cong \Ext^{\ell}_R (\kappa (\fp),R / \fp),
$$
the last of which must be zero since $\ell <0$. 
By this contradiction we complete the proof.
\end{proof}

Let $a, b\in \mathbb{Z}\cup \{ \pm\infty\}$ with $a\leq b$.  We denote by $\cD^{[a,b]}$ for the full subcateogy of $\cD$ consisting of complexes 
$X$ such that $H^i(X)=0$ for $i\notin [a, b]$ (cf. \cite[Notation 13.1.11]{KS}). Moreover, we write $\cD^{[a,b]}_{\rm fg}=\cD^{[a,b]}\cap \cD_{\rm fg}$. 
For $X\in \cD$, we denote by $\inf X$ the infimum of the set $\Set{i\in \mathbb{Z} | H^i(X)\neq 0}$.

\begin{corollary}\label{negative vanishing 2}
Let $X\in \cD_{\rm fg}$ and $W$ be a subset of $\Spec R$. We assume that $\dim W$ is finite. Then we have $\inf X\leq \inf \gamma_WX$.
\end{corollary}

\begin{proof}
We may assume that $\inf X>-\infty$.
Thus it is enough to show that if $X\in \cD^{[0,+\infty]}_{\rm fg}$, then $\gamma_WX\in \cD^{[0,+\infty]}$. To see this, setting $n=\dim W$, we consider the triangle $\sigma_{\leq n}X\to X\to \sigma_{> n}X\to (\sigma_{\leq n}X)[1]$.
Since $\sigma_{\leq n}X\in \cD^{[0,n]}_{\rm fg}$, by Proposition \ref{negative vanishing}, we can show that $\gamma_W (\sigma_{\leq n}X)\in \cD^{[0,+\infty]}$. Furthermore, it follows from Corollary \ref{[a,b]} that $\gamma_W (\sigma_{> n}X)\in \cD^{[0,+\infty]}$.
Thus we can conclude that $\gamma_WX\in \cD^{[0,+\infty]}$ by the triangle.
\end{proof}

Let $X\in \cD$. For an ideal $\fa$ of $R$, 
we define the $\fa$-depth of $X$ as
$\depth(\fa, X)=\inf \RHom_R(R/\fa,X)$.
Let $\boldsymbol{x}=\{x_1,\ldots,x_n\}$ be a system of generators of $\fa$.
For each $x_i$, 
$K(x_i)$ denotes the complex $(0\to R\xrightarrow{x_i}R\to 0)$ concentrated in degrees $-1$ and $0$.
The Koszul complex with respect to $\boldsymbol{x}$ is the complex $K(\boldsymbol{x})=K(x_1)\otimes_R\cdots\otimes_RK(x_n)$.
By \cite[Theorem 2.1]{FI}, it holds that
$\depth(\fa, X) =\inf\Hom_R(K(\boldsymbol{x}),X)=\inf \RGamma_{V(\fa)}X$.

If $W$ is a specialization-closed subset of $\Spec R$, 
then we define the $W$-depth of $X$, which we denote by $\depth(W, X)$, as the infimum of the set of values $\depth(\fa, X)$ for all ideals $\fa$ with $V(\fa)\subseteq W$. It is easily seen that $\depth(W, X)=\inf \RGamma_WX$.

\begin{proposition}\label{depth corollary}
Let $W$ be a subset of $\Spec R$, and assume that $\dim W$ is finite.
Let $X\in \cD_{\rm fg}$.
Then we have 
$\depth(\overline{W}^s, X)\leq \inf \gamma_W X$.
\end{proposition}

\begin{proof}
By Remark \ref{inclusion} (i), it holds that
$\gamma_WX\cong \RGamma_{\overline{W}^s}(\gamma_WX)$.
Therefore we have $$\inf \gamma_WX=\inf \RGamma_{\overline{W}^s}(\gamma_WX)=\depth(\overline{W}^s, \gamma_WX).$$ 
Hence it remains to show that 
$\depth(\overline{W}^s, X)\leq \depth(\overline{W}^s, \gamma_WX)$.
Let $\fa$ be an ideal of $R$ with $V(\fa)\subseteq\overline{W}^s$ and $\boldsymbol{x}=\{x_1,\ldots,x_n\}$ be a system of generators of $\fa$.
Since $K(\boldsymbol{x})$ is a bounded complex of finitely generated free $R$-modules, it follows from Corollary \ref{negative vanishing 2} that
$$\inf\Hom_R(K(\boldsymbol{x}),X)\leq
\inf \gamma_W\Hom_R(K(\boldsymbol{x}),X)=\inf\Hom_R(K(\boldsymbol{x}),\gamma_WX).$$
Thus we have $\depth(\fa,X)\leq  \depth(\fa,\gamma_WX)$. 
Hence it holds that $\depth(\overline{W}^s,X)\leq  \depth(\overline{W}^s,\gamma_WX)$ by definition.
\end{proof}\vspace{1mm}

This proposition states that an inequality
$$\inf \RGamma_{\overline{W}^s}X\leq \inf \gamma_{W} X.$$
holds for a subset $W$ of $\Spec R$ with $\dim W<+\infty$ and $X\in \cD_{\rm fg}$.\\
 
Let $X\in \cD_{\rm fg}$. 
We denote by $\dim X$ the supremum of the set $\Set{\dim H^i(X)+i |  i\in \mathbb{Z}}$, see \cite[\S 3]{F}.
Note that $\dim X[1]=\dim X-1$. 
Let $D_R$ be a dualizing complex of $R$ with $D_R\in \cD^{[0,d]}_{\rm fg}$, where $d=\dim R$.
In the proof of the next proposition, we use a basic fact that 
$X^{\dag}\in \cD_{\rm fg}^{[d-n,+\infty]}$, where $n=\dim X$.
To show this, we first suppose that $X$ is a finitely generated $R$-module. 
Then it is straightforward to see that $X^\dag\in \cD^{[d-n,d]}_{\rm fg}\subset \cD_{\rm fg}^{[d-n,+\infty]}$.
Next, supppse that $X$ is any complex of $\cD_{\rm fg}$. Notice that it suffices to treat the case that $n=\dim X=0$.
Then, using the triangle 
$\sigma_{\leq -d}X\to X\to \sigma_{>-d}X\to (\sigma_{\leq -d}X)[1]$, one can deduce that $X^{\dag}\in \cD_{\rm fg}^{[d,+\infty]}$ as desired. 
This fact is essentially proved in \cite[Proposition 3.14 (d)]{F}. 

\begin{proposition}\label{vanishing dim M}
Assume that $R$ admits a dualizing complex. 
Let $W$ be a subset of $\Spec R$ and $X\in \cD_{\rm fg}$. Then $ H^i ( \gamma_W X )=0$ for all  $i> \dim X$.
\end{proposition}

\begin{proof}
Let $D_R$ be a dualizing complex with $D_R\in \cD^{[0,d]}_{\rm fg}$, where $d=\dim R$.
By Corollary \ref{[a,b]}, $\gamma_W D_R$ is isomorphic to a bounded complex $I$ of injective $R$-modules with $I^i=0$ for $i>d$.
Then, by Corollary \ref{LDT}, there are isomrophisms 
$$
\gamma_W X \cong  \RHom_R(X^\dag, \gamma_W D_R)\cong \Hom_R(X^\dag, I).
$$
Therefore each element of $H^i(\gamma_WX) \cong \Hom_\cD (X^\dag,I[i])$ is represented by a chain map from   
$X^\dag$  to  $I[i]$.
Moreover, setting $n=\dim X$, we have $X^\dag\in \cD^{[d-n, +\infty]}_{\rm fg}$ by the above-mentioned fact.
Hence it holds that $H^i(\gamma_WX)=0$ for all $i>n=\dim X$.
 \end{proof}

\color{black}
We sum up Proposition \ref{depth corollary} and Proposition \ref{vanishing dim M} in the following theorem.

\begin{theorem}\label{GVT}
Assume that $R$ admits a dualizing complex. 
Let $W$ be a subset of $\Spec R$. 
If $X\in \cD_{\rm fg}$, then $ H^i(\gamma_WX)=0$ unless $\depth(\overline{W}^s, X) \leq i \leq \dim X$.
\end{theorem}


\bibliographystyle{amsplain}

\end{document}